\documentclass{amsart}
\usepackage{latexsym,amsmath,amssymb}
\newtheorem{thm}{Theorem}[section]
\newtheorem{cor}[thm]{Corollary}
\newtheorem{lem}[thm]{Lemma}

\newtheorem{pro}[thm]{Proposition}
\theoremstyle{definition}\newtheorem{defn}[thm]{Definition}
\theoremstyle{remark}
\newtheorem{rem}[thm]{Remark}
\numberwithin{equation}{section}

\begin{document}
	
\title
[]
{Ascent and descent of WCT operators on Orlicz spaces}

\author{\sc\bf S. Nojoumi and S. Shamsigamchi }
\address{\sc }

\email{s.shamsi@pnu.ac.ir}
%\email{}
%\email{}

\address{ Department of Mathematics, Payame Noor University (PNU), P. O. Box: 19395-3697, Tehran- Iran}



%______________________________________________________________________

\begin{abstract}
In this paper we are concerned with weighted conditional type(WCT) operators on Orlicz spaces. We prove that all WCT operators have finite ascent. Also, we provide some sufficient conditions for WCT operators to have finite descent. As a consequence we find some decompositions for Orlicz space $L^{\Phi}(\mu)$. In the sequel we discuss power bounded WCT operators and some results on their Cesaro boundedness.

\end{abstract}

\maketitle

\section{ \sc\bf Introduction and Preliminaries }
Let $X$ be a linear space and $T:X\longrightarrow X$  be a linear operator with domain $\mathcal{D}(T)$ and range $\mathcal{R}(T)$ in $X$.  The null space of the iterates of $T$, $T^n$, is denoted by   $\mathcal{N}(T^n)$, and we know that the null spaces of  $T^n$'s form an increasing chain of subspaces
$\{0\}=\mathcal{N}(T^0)\subset \mathcal{N}(T)\subset \mathcal{N}(T^2)\subset\ldots$. Also  the ranges  of iterates of $T$ form a nested chain of subspaces $X=\mathcal{R}(T^0)\supset \mathcal{R}(T)\supset \mathcal{R}(T^2)\supset\ldots$.
  Note that if  $\mathcal{N}(T^k)$ coincides with $\mathcal{N}(T^{k+1})$ for some $k$,  it coincides with all $\mathcal{N}(T^n)$ for $n>k$.
The smallest non-negative integer $k$ such that $\mathcal{N}(T^k)=\mathcal{N}(T^{k+1})$ is called the \textit{ascent} of $T$ and denotes by $\alpha(T)$. If there is no such $k$, then  we set $\alpha(T)=\infty$. Also  if $\mathcal{R}(T^k)=\mathcal{R}(T^{k+1})$, for some  non-negative integer $k$, then  $\mathcal{R}(T^n)=\mathcal{R}(T^k)$ for all $n>k$.
 The smallest non-negative integer $k$ such that $\mathcal{R}(T^k)=\mathcal{R}(T^{k+1})$ is called  \textit{descent} of $T$
and denotes by $\delta(T)$. We set $\delta(T)=\infty$ when there is no such $k$. When ascent and descent of an operator are finite, then they are equal and the linear space $X$ can be decomposed into the direct sum of the null and range spaces of a suitable  iterates of $T$. The ascent and descent of an operator can be used to characterize when an operator can be broken into a nilpotent piece and an invertible one; see, for example, \cite{abr, Tay}. For some results on ascent and descent of bounded operators in general setting see, for example, \cite{tay,yoo}.\\
The operator $T$ is called power bounded if the norms of $T^k$, $k\geq0$, are uniformly bounded $(\sup_{k}\|T^k\|<\infty)$, and Cesaro bounded if the Cesaro means $A_n(T)=n^{-1}\sum^{n-1}_{i=0}T^{i}$ are uniformly bounded.\\
Here we recall the concepts on Orlicz spaces. A function $\Phi:\mathbb{R}\rightarrow [0,\infty]$ is called a \textit{Young function}  if $\Phi$ is   convex, even,  and  $\Phi(0)=0$; we will also assume that $\Phi$ is neither identically zero nor identically infinite on $(0,\infty)$. The fact that $\Phi(0)=0$, along with the convexity of $\Phi$, implies that $\lim_{x\rightarrow 0^+}\Phi(x)=0$; while  $\Phi\neq 0$, again along with the convexity of $\Phi$, implies that   $\lim_{x\rightarrow\infty}\Phi(x)=\infty$. We set
$a_{\Phi}:=\sup\{x\geq0:\Phi(x)=0\}$
 and
$
b_{\Phi}:=\sup\{x>0:\Phi(x)<\infty\}.
$
Then it can be checked that  $\Phi$ is continuous and nondecreasing on
$[0,b_{\Phi})$ and strictly increasing on $[a_{\Phi},b_{\Phi})$. We also assume the left-continuity of the function $\Phi$ at $b_\Phi$, i.e. $\lim_{x\rightarrow b_\Phi^-} \Phi(x)=\Phi(b_\Phi)$.

 To each Young  function $\Phi$ is  associated another
 convex function $\Psi:\mathbb{R}\rightarrow[0,\infty)$ with similar properties,  defined by
$$
\Psi(y)=\sup\{x|y|-\Phi(x):x\geq0\} \quad (y\in\mathbb{R}).
$$
The function $\Psi$ is called the  \textit{function  complementary} to $\Phi$ in the sense of Young. Also, for any measurable function $f$ on the measure space $(X,\Sigma, \mu)$ we set $I_{\Phi}(f)=\int_{X}\Phi(f)d\mu$.
 Any pair of complementary functions $(\Phi,\Psi)$ satisfies Young's inequality $xy\leq \Phi(x)+\Psi(y)\,\, (x,y\geq 0)$.

The generalized inverse of the Young function $\Phi$ is defined by
$$
\Phi^{-1}(y)=\inf \{ x\geq 0: \Phi(x)> y\} \quad (y\in [0,\infty)).
$$
Notice that if $x\geq0$, then $\Phi\big(\Phi^{-1}(x)\big)\leq x$,
and if $\Phi(x)<\infty$, we also have $x\leq\Phi^{-1}\big(\Phi(x)\big)$. There are equalities in either case when $\Phi$ is a Young function vanishing only at zero and taking only finite values.
Also, if $(\Phi,\Psi)$ is a pair of complementary
Young functions, then
\begin{equation}\label{12}
x<\Phi^{-1}(x)\Psi^{-1}(x)\leq 2x
\end{equation}
for all $x\geq0$ (Proposition 2.1.1(ii) \cite{raor}).

By an $N$-\textit{function} we mean a  Young function vanishing only at zero, taking only finite values, and such that $\lim_{x\rightarrow\infty}\Phi(x)/x=\infty$ and $\lim_{x\rightarrow 0^+}\Phi(x)/x=0$. Note that then $a_\Phi=0,$ $b_\Phi=\infty$, and, as we said above,   $\Phi$ is continuous and strictly increasing on $[0,\infty)$. Moreover, a function complementary to an $N$-function is again an $N$-function.

A Young function $\Phi$ is said to satisfy the
$\Delta_{2}$-condition at $\infty$ if $\Phi(2x)\leq
K\Phi(x) \; ( x\geq x_{0})$  for some constants
$K>0$ and $x_0>0$. A Young function $\Phi$ satisfies the
$\Delta_{2}$-condition globally if $\Phi(2x)\leq
K\Phi(x) \; ( x\geq 0)$  for some
$K>0$.

A Young function $\Phi$ is said to satisfy the
$\Delta'$-condition (respectively, the $\nabla'$-condition) at $\infty$, if there exist  $ c>0$
(respectively, $b>0$) and $x_0>0$ such that
$$
\Phi(xy)\leq c\,\Phi(x)\Phi(y) \quad (x,y\geq x_{0})
$$
$$
(\mbox{respectively, }  \Phi(bxy)\geq \Phi(x)\Phi(y) \quad ( x,y\geq x_{0})).
$$
If $x_{0}=0$, these conditions are said to hold
globally. Notice that if $\Phi\in \Delta'$, then  $\Phi\in
\Delta_{2}$ (both at $\infty$ and globally).

 Let $\Phi, \Psi$ be Young
functions. Then $\Phi$ is called stronger than $\Psi$ at $\infty$, which is denoted by $\Phi\mathrel{\overset{\makebox[0pt]{\mbox{\normalfont\scriptsize\sffamily $\ell$ }}}{\succ}}\Psi$ [or $\Psi\mathrel{\overset{\makebox[0pt]{\mbox{\normalfont\scriptsize\sffamily $\ell$}}}{\prec}}\Phi$], if
$$
\Psi(x)\leq\Phi(ax)\quad (x\geq x_0)
$$
for some $a\geq0$ and $x_0>0$; if $x_0=0$, this condition is
said to hold globally and is then denoted by  $\Phi\mathrel{\overset{\makebox[0pt]{\mbox{\normalfont\scriptsize\sffamily $a$}}}{\succ}}\Psi$ [or $\Psi\mathrel{\overset{\makebox[0pt]{\mbox{\normalfont\scriptsize\sffamily $a$}}}{\prec}}\Phi$].

%We record  the following observation for later use.
%
%\begin{lem}\label{l2} If $\Phi, \Psi, \Theta$ are Young functions vanishing only at zero, taking only finite values, and such that
%$$
%\Phi(xy)\leq\Psi(x)+\Theta(y)\quad (x,y\geq0),
%$$
%then $\Psi\mathrel{\overset{\makebox[0pt]{\mbox{\normalfont\scriptsize\sffamily $\ell$}}}{\nprec}}\Phi$, and hence also  $\Psi\mathrel{\overset{\makebox[0pt]{\mbox{\normalfont\scriptsize\sffamily $a$}}}{\nprec}}\Phi$.
%\end{lem}

For a given complete $\sigma$-finite measure space $(X, \Sigma, \mu)$, let $L^0(\Sigma)$  be the linear space of  equivalence classes of $\Sigma$-measurable real-valued functions on $X$, that is, we identify functions equal $\mu$-almost everywhere on $X$. The support $S(f)$ of a
measurable function $f$ is defined by $S(f):=\{x\in X : f(x)\neq
0\}$. For a given Young function  $\Phi$, the space
$$
L^{\Phi}(\mu)=\left\{f\in L^0(\Sigma):\exists k>0,
\int_X\Phi(kf)d\mu<\infty\right\}
$$
is called an Orlicz space. Define the functional

$$N_{\Phi}(f)=\inf \{k>0:\int_{X}\Phi(\frac{f}{k})d\mu\leq 1\}.$$
 $(L^{\Phi}(\mu), N_{\Phi}(.))$  is a normed linear space. If a.e. equal functions are identified, then $(L^{\Phi}(\mu), N_{\Phi}(.))$ is a Banach space, the basic measure space $(X,\Sigma,\mu)$ is unrestricted.

$$
\|f\|_{\Phi}=\inf\left\{k>0:\int_X\Phi(f/k)d\mu\leq1\right\}.
$$
 The couple $(L^{\Phi}(\Sigma), \|\cdot\|_{\Phi})$ is called the Orlicz space generated by a Young function $\Phi$.
 Let $\Phi(x)=|x|^p/p$ with $1<p<\infty$; $\Phi$ is then a Young function and $\Psi(x)=|x|^{p'}/p'$, with $1/p+1/p'=1$, is the  Young function complementary to $\Phi$.
 Thus, with this function $\Phi$ we retrieve the classical Lebesgue space $L^p(\Sigma)$, i.e. $L^\Phi(\Sigma)=L^p(\Sigma)$.

Recall that an atom of the measure space $(X,\Sigma,\mu)$ is a set $A\in\Sigma$ with $\mu(A)>0$ such that if
$F\in\Sigma$ and $F\subset A$, then either $\mu(F)=0$ or
$\mu(F)=\mu(A)$. A measure space $(X,\Sigma,\mu)$ with no atoms is
called a non-atomic measure space. It is well-known  that if $(X, \Sigma, \mu)$ is a $\sigma$-finite measure space, then for every measurable real-valued function $f$ on $X$ and every atom $A$, there is a unique scalar, denoted  by $f(A)$,  such that $f=f(A)\,\, \mu$-a.e. on $A$. Also, if $(X, \Sigma, \mu)$ is a $\sigma$-finite measure space that fails to be non-atomic, there is a non-empty countable set  of pairwise disjoint atoms $\{A_n\}_{n\in\mathbb{N}}$  with the property that ${B}:=X\setminus\bigcup_{n\in\mathbb{N}}A_n$ contains no atoms \cite{z}.\\
 %Here we recall the next lemma that is a key tool in our investigations.
%\begin{lem}\label{p1} Let $\Phi, \Psi$ be Young functions such that  $\Psi\mathrel{\overset{\makebox[0pt]{\mbox{\normalfont\scriptsize\sffamily $\ell$}}}{\nprec}}\Phi$. If $E$ is a
%non-atomic $\Sigma$-measurable set with positive measure, then there exists $f\in L^{\Phi}(\Sigma)$ such that $f_{|_E}\notin L^{\Psi}(E)$. c.f. \cite{ceh}.
%\end{lem}
For a sub-$\sigma$-finite algebra $\mathcal{A}\subseteq\Sigma$, the
conditional expectation operator associated with $\mathcal{A}$ is
the mapping $f\rightarrow E^{\mathcal{A}}f$, defined for all
non-negative, measurable function $f$ as well as for all $f\in
L^1(\Sigma)$ and $f\in L^{\infty}(\Sigma)$, where
$E^{\mathcal{A}}f$, by the Radon-Nikodym theorem, is the unique
$\mathcal{A}$-measurable function satisfying
$$\int_{A}fd\mu=\int_{A}E^{\mathcal{A}}fd\mu, \ \ \ \forall A\in \mathcal{A} .$$
As an operator on $L^{1}({\Sigma})$ and $L^{\infty}(\Sigma)$,
$E^{\mathcal{A}}$ is idempotent and
$E^{\mathcal{A}}(L^{\infty}(\Sigma))=L^{\infty}(\mathcal{A})$ and
$E^{\mathcal{A}}(L^1(\Sigma))=L^1(\mathcal{A})$. Thus it can be
defined on all interpolation spaces of $L^1$ and $L^{\infty}$ such
as, Orlicz spaces \cite{besh}. If there is no possibility of
confusion, we write $E(f)$ in place of $E^{\mathcal{A}}(f)$. This
operator will play a major role in our work and we list here some
of its useful properties:

\vspace*{0.2cm} \noindent $\bullet$ \  If $g$ is
$\mathcal{A}$-measurable, then $E(fg)=E(f)g$.

\noindent $\bullet$ \ $\varphi(E(f))\leq E(\varphi(f))$, where
$\varphi$ is a convex function.

\noindent $\bullet$ \ If $f\geq 0$, then $E(f)\geq 0$; if $f>0$,
then $E(f)>0$.

\noindent $\bullet$ \ For each $f\geq 0$, $S(f)\subseteq S(E(f))$,
where  $S(f)=\{x\in X; f(x)\neq 0\}$.\\
A detailed discussion and verification of
most of these properties may be found in \cite{rao}.

Let $f\in L^{\Phi}(\Sigma)$.  It is not difficult to see that $\Phi(E(f))\leq
E(\Phi(f))$ and so by some elementary computations we get that $\|E(f)\|\leq \|f\|$ i.e, $E$ is a
contraction on the Orlicz spaces.
As we defined in \cite{ye}, we say that the pair $(E, \Phi)$ satisfies the generalized conditional-type
H\"{o}lder-inequality (or briefly GCH-inequality) if there exists some positive constant $C$
such that for all $f\in L^{\Phi}(\mu)$ and $g\in
L^{\Psi}(\mu)$ we have
$$E(|fg|)\leq C \Phi^{-1}(E(\Phi(|f|)))\Psi^{-1}(E(\Psi(|g|))),$$
where $\Psi$ is the complementary Young function of $\Phi$. There are many examples of pairs $(E, \Phi)$ that satisfy GCH-inequality in \cite{ye}.

Finally in the following we give another key lemma that is important in our investigation. The proof is an easy exercise.
\begin{lem}\label{l1n} If $\Phi$ is a Young function and $f$ is a $\Sigma$-measurable function such that $E(f)$ and $E(\Phi(f))$ are defined, then $S(E(f))=S(E(\Phi(f)))$.
\end{lem}
We keep the above notations throughout the paper.\\
Weighted conditional type operators have been studied by many mathematicians in recent years \cite{e1,ye,y1,ej,gb}. For the importance of WCT operators we refer the interested readers to \cite{dhd, do,mo,lam}. In this paper we are going to investigate ascent and descent of weighted conditional type (WCT) operators on Orlicz spaces. To this end, first we show that all WCT operators have finite ascent. Also, we provide some sufficient conditions for them to have finite descent. As a consequence, we find some decompositions for Orlicz space $L^{\Phi}(\mu)$. Finally we find some results on power bounded WCT operators and Cesaro bounded WCT operators.\\

\section{ Main Results}
In this section we determine which hypothesis let us get that WCT operator $T=M_wEM_u$ is a bounded operator on the Orlicz space $L^{\Phi}(\mu)$. After that we will discuss the conditions under which the bounded operator $T$ has finite ascent and descent. Also, we will find some sufficient condition for $T$ to have closed range. Some other results will be obtained. Here we define  weighted conditional type operators on Orlicz spaces.

\begin{defn} Let $\Phi$ be a Young function and $u,w:X\rightarrow \mathbb{C}$ be a measurable function on the measure space $(X,\Sigma,\mu)$. The weighted conditional type operator (WCT operator) from $L^{\Phi}(\Sigma)$ into $L^{\Phi}(\Sigma)$ is defined by $wEM_u(f)=wE(uf)$ for every $f\in L^{\Phi}(\Sigma)$ such that $wE(uf)\in L^{\Phi}(\mathcal{A})$.
 \end{defn}
In the next theorem we give a condition under which the WCT operator $T=M_wEM_u$ is bounded on the Orlicz space $L^{\Phi}(\mu)$.

 \begin{thm}
 Let $(\Phi,\Psi)$ be a pair of complementary Young's functions such that satisfies GCH-inequality and $T=M_wEM_u$ be WCT operator. If $w\Psi^{-1}(E(\Psi(u)))\in L^{\infty}(\Sigma)$, then $T$ is a bounded operator on $L^{\Phi}(\mu)$.
\end{thm}
 \begin{proof}
Since $(\Phi,\Psi)$ satisfies GCH-inequality, then there exists $C>0$ such that for all $f\in L^{\Phi}(\mu)$ and $g\in
L^{\Psi}(\mu)$ we have
$$E(|fg|)\leq C \Phi^{-1}(E(\Phi(|f|)))\Psi^{-1}(E(\Psi(|g|))).$$
Let $M=\|w\Psi^{-1}(E(\Psi(u)))\|_{\infty}$. For each $f\in L^{\Phi}(\mu)$ we have
\begin{align*}
\int_X\Phi(\frac{wE(uf)}{CMN_{\Phi}(f)}d\mu&\leq \int_X\Phi(\frac{wC \Phi^{-1}(E(\Phi(\frac{|f|}{N_{\Phi}(f)})))\Psi^{-1}(E(\Psi(|u|)))}{CM}d\mu\\
&\leq \int_X\Phi(\Phi^{-1}(E(\Phi(\frac{|f|}{N_{\Phi}(f)})))d\mu\\
&=\int_X\Phi(\frac{|f|}{N_{\Phi}(f)})d\mu\\
&\leq 1.
\end{align*}
By these observations we get that $N_{\Phi}(wE(uf))\leq CM N_{\Phi}(f)$, so $T$ is bounded.
 \end{proof}
Now in the next Lemma we see that for every $n\in \mathbb{N}$, $T^n$ is again a WCT operator.
 \begin{lem}\label{l2.3} Let $n\in \mathbb{N}$ and $T=M_wEM_u$ be a bounded operator on $L^{\Phi}(\Sigma)$. Then we have
$$T^n=M_{(E(uw))^{n-1}}M_{w}EM_{u}.$$
\end{lem}
 \begin{proof}
 By induction and straight forward calculations one can get the proof.
 \end{proof}
In the sequel we find that every bounded WCT operator has finite ascent.
 \begin{pro}\label{p2.4}
 Let the WCT operator $T=M_wEM_u$ be bounded on the Orlicz space $L^{\Phi}(\Sigma)$. The for every $n\in \mathbb{N}$, $\mathcal{N}(T^2)=\mathcal{N}(T^{n+2})$. As a result $\alpha(T)\leq 2$.
  \end{pro}
 \begin{proof} Let $n\in \mathbb{N}\setminus \{1\}$, then by Lemma \ref{l2.3} we have
 $$T^n(f)=(E(uw))^{n-1}wE(uf)$$
 and specially for $n=2$, $T^2(f)=(E(uw))wE(uf)$,  for $f\in L^{\Phi}(\mu)$. It is clear that $T^2(f)=0$ if and only if $T^n(f)=0$, for all $n\geq 2$. So
 $\mathcal{N}(T^2)=\mathcal{N}(T^{n})$.

 \end{proof}
 For our main results we need to find some sufficient conditions for closedness of range of WCT operator $T=M_wEM_u$. So in the next proposition we provide some conditions under which $T$ has closed range.

 \begin{pro}\label{p2.5} Let $(\Phi,\Psi)$ a pair of complementary Young's functions that satisfies GCH-inequality, WCT operator $T$ be bounded on $L^{\Phi}(\Sigma)$, $\mu(B)=0$ and the set $$H=\{n\in \mathbb{N}: w(A_n)\Psi^{-1}(E(\Psi(u)))(A_n)\neq 0\}$$ be finite. Then $T$ has closed range.
\end{pro}
 \begin{proof}
Since $(\Phi,\Psi)$ satisfies GCH-inequality, then there exists $C>0$ such that for every $f\in L^{\Phi}(\mu)$,
$$wE(uf)\leq wC \Phi^{-1}(E(\Phi(|f|)))\Psi^{-1}(E(\Psi(|u|))),$$
so $$S(wE(uf))\subset S(w)\cap S(\Psi^{-1}(E(\Psi(|u|))))=S(w\Psi^{-1}(E(\Psi(|u|)))).$$
Moreover, by our assumptions we have
$$S_0:=S(w\Psi^{-1}(E(\Psi(|u|))))=\cup\{A_n:n\in H\}.$$
Hence we get that $\mathcal{R}(T)\subset L^{\Phi}(S_0, \Sigma_{S_0}, \mu_{S_0})$. Since $H$ is finite and $\Sigma$-atoms are disjoint, then we get that $L^{\Phi}(S_0, \Sigma_{S_0}, \mu_{S_0})$ is finite dimensional and therefore $\mathcal{R}(T)$ is finite dimensional. Consequently, $\mathcal{R}(T)$ is closed.
 \end{proof}
 Here we obtain a condition under which $T^n$ has closed range for all $n\in \mathbb{N}$.
 \begin{thm}\label{t2.6}
 Let $(\Phi,\Psi)$ a pair of complementary Young's functions that satisfies GCH-inequality, WCT operator $T$ be bounded on $L^{\Phi}(\Sigma)$, $\mu(B)=0$ and the set $$H_k=\{n\in \mathbb{N}: E(wu)^{k}(A_n).w(A_n)\Psi^{-1}(E(\Psi(u)))(A_n)\neq 0\}$$ be finite, for some $k\in\mathbb{N}$. Then $T^n$ has closed range, for all $n\in\mathbb{N}\setminus\{1\}$. Consequently, if
 $$H=H_0=\{n\in \mathbb{N}: E(wu)^{k}(A_n).w(A_n)\Psi^{-1}(E(\Psi(u)))(A_n)\neq 0\}$$
 is finite, then $T^n$ has closed range, for all $n\in\mathbb{N}$.
 \end{thm}
 \begin{proof}
As we had in the Proposition \ref{p2.4} for each $n\in \mathbb{N}$, $T^n=M_{w_n}EM_u$, in which $w_n=wE(uw)^{n-1}$, so $T^n$ is also a WCT operator. Moreover, by our assumptions, for every $m,n\in \mathbb{N}$, with $m\neq n$ we have $H_n=H_m$. Hence if $H_k$ is finite for some $k\in \mathbb{N}$, then by Proposition \ref{p2.5} we get that $T^n$ has closed range for all $n\in \mathbb{N}$. In addition if $H$ is finite then $T^n$ has closed range for all $n\in \mathbb{N}$, because $H_k\subseteq H$, for every $k\in \mathbb{N}$.
 \end{proof}
In the following we have some consequences on closedness of range of $T$ and some decompositions for $L^{\Phi}(\mu)$.
 \begin{rem}\label{r2.7}
 Let WCT operator $T=M_wEM_u$ be bounded on $L^{\Phi}(\Sigma)$. If $\mathcal{R}(M_{E(wu)^{n-1}}T)$ is closed for some $n>2$ or $\mathcal{R}(M_{E(wu)^j}T)+\mathcal{N}(M_{E(wu)^k}T)$ is closed for some positive integers with $j+k=n$, then for $\mathcal{R}(T^n)=\mathcal{R}(M_{E(wu)^{n-1}}T)$ is closed for all $n\geq 2$ and $\mathcal{R}(M_{E(wu)^j}T)+\mathcal{N}(M_{E(wu)^k}T)$ is closed for all $j+k\geq 2$. Moreover, $L^{\Phi}(\Sigma)=\mathcal{R}(M_{E(wu)}T)+\mathcal{N}(M_{E(wu)}T)$.
 \end{rem}
 \begin{proof}
As we proved in Proposition \ref{p2.4}, $\alpha(T)\leq 2$. Therefore by Theorem 2.1 of \cite{gz} we get the proof.
 \end{proof}
 Now by mixing Theorem \ref{t2.6} and Remark \ref{r2.7} we get the next corollary.
 \begin{cor}
 Let $(\Phi,\Psi)$ a pair of complementary Young's functions that satisfies GCH-inequality, WCT operator $T$ be bounded on $L^{\Phi}(\Sigma)$, $\mu(B)=0$ and the set $$H_k=\{n\in \mathbb{N}: E(wu)^{k}(A_n).w(A_n)\Psi^{-1}(E(\Psi(u)))(A_n)\neq 0\}$$ be finite, for some $k\in\mathbb{N}$. Then  $\mathcal{R}(M_{E(wu)^j}T)+\mathcal{N}(M_{E(wu)^k}T)$ is closed for all $j+k\geq 2$. Moreover, $L^{\Phi}(\Sigma)=\mathcal{R}(M_{E(wu)}T)+\mathcal{N}(M_{E(wu)}T)$.
\end{cor}
 \begin{pro}\label{p2.9}
 If $w\Psi^{-1}(E(\Psi(u)))\in L^{\infty}(\Sigma)$, then the following hold:\\

 a) The sequence $\{\|E(wu)\|_{\infty}\}_{n}$ is uniformly bounded if and only if $\|E(wu)\|_{\infty}\leq 1$.\\

 b) The WCT operator $T$ is power bounded on the Orlicz space $L^{\Phi}(\Sigma)$ if and only if $|E(wu)|<1$ on $S(\Phi^{-1}(E(\Phi(w))))\cap S(\Psi^{-1}(E(\Psi(u))))$.
 \end{pro}
 \begin{proof}
a) For the proof of this part one can see Theorem 2.5, part (a), \cite{e1}.\\
b) Let $T$ be power bounded. Then there exists $D>0$ such that
$$\|T^n\|=\|M_{E(wu)^{n-1}}T\|\leq D, \ \ \ \ \forall n\in \mathbb{N}.$$
As we discussed before, since $(\Phi,\Psi)$ satisfies GCH-inequality, then for every $f\in L^{\Phi}(\mu)$,
$$T(f)=wE(uf)\leq wC \Phi^{-1}(E(\Phi(|f|)))\Psi^{-1}(E(\Psi(|u|))),$$
for some $C>0$. Hence
\begin{align*}
S(wE(uf))&\subset S(w)\cap S(\Psi^{-1}(E(\Psi(|u|))))\\
&=S(w\Psi^{-1}(E(\Psi(|u|))))\\
&\subseteq S(\Psi^{-1}(E(\Psi(|w|))))S(\Psi^{-1}(E(\Psi(|u|)))).
\end{align*}
Also, by GCH-inequality we have
$$S(E(wu)\subseteq S(\Psi^{-1}(E(\Psi(|w|))))S(\Psi^{-1}(E(\Psi(|u|)))).$$
Moreover, $S(T^n(f))=S(E(wu)^{n-1})\cap S(Tf)$. By these observations we get that if $\|E(wu)\|>1$ on a set of positive measure, then the non-zero WCT operator $T$ can't be power bounded.\\
Conversely, let
$$|E(wu)|<1 \ \ \ \text{on} \  \ S(\Phi^{-1}(E(\Phi(w))))\cap S(\Psi^{-1}(E(\Psi(u)))).$$ Then $\|E(wu)^n\|\leq C$ or some $C>0$ and so for every $n\in \mathbb{N}$,
$\|T^n\|\leq C\|T\|$. Thus $T$ is power bounded.
\end{proof}
In the sequel we provide some conditions under which WCT operators have finite descent.

 \begin{thm}
 If $w\Psi^{-1}(E(\Psi(u)))\in L^{\infty}(\Sigma)$, and $E(uw)$ is bounded away from zero, then for each $n\in \mathbb{N}$, $\mathcal{R}(T^{n+2})=\mathcal{R}(T^{2})$ and so $T$ has finite descent.
  \end{thm}
 \begin{proof}
It is clear that for any $n\in\mathbb{N}$, $\mathcal{R}(T^{n+2})\subseteq \mathcal{R}(T^2)$. For the converse, let $g\in \mathcal{R}(T^2)$. Then there exists $f\in L^{\Phi}(\mu)$ such that $T^2(f)=E(wu)T(f)$. Also, since $E(uw)$ is bounded away from zero, then $\frac{1}{E(wu)}\chi_{S(E(wu))}\in L^{\infty}(\mu)$. Hence $\frac{1}{E(wu)^n}\chi_{S(E(wu))}f\in L^{\Phi}(\mu)$.
\begin{align*}
g&=E(wu)T(f)\\
&=\frac{1}{E(wu)^n}\chi_{S(E(wu))}E(wu)^{n+1}T(f)\\
&=E(wu)^{n+1}T(\frac{1}{E(wu)^n}\chi_{S(E(wu))}f)\\
&=T^{n+2}(\frac{1}{E(wu)^n}\chi_{S(E(wu))}f).
\end{align*}
This implies that $g\in \mathcal{R}(T^{n+2})$. Consequently we have $\mathcal{R}(T^{n+2})=\mathcal{R}(T^{2})$.
 \end{proof}

\begin{cor}
Let $w\Psi^{-1}(E(\Psi(u)))\in L^{\infty}(\Sigma)$, $E(wu)$ be bounded away from zero and $T=M_wEM_u$ be WCT operator on the Orlicz spaces $L^{\Phi}(\Sigma)$. Then $\delta(T)\leq 2$.
\end{cor}
Till now we have obtained that bounded WCT operators $T=M_wEM_u$ have finite ascent with $\alpha(T)\leq 2$ and also under a weak conditions have finite descent with $\delta(T)\leq 2$. In the next Proposition some more results affected by finite ascent and descent.
\begin{pro}\label{p2.12}
Let $T=M_wEM_u$ be bounded WCT operator on the Orlicz spaces $L^{\Phi}(\Sigma)$. Then $\mathcal{R}(T^2)\cap \mathcal{N}(T^m)=\{0\}$, for every $m\geq 1$. Also, if $E(wu)$ be bounded away from zero, then $L^{\Phi}(\mu)=\mathcal{R}(T^n)+\mathcal{N}(T^2)$, for some (equivalently, all) $n\geq 1$.
\end{pro}
\begin{proof}
As it is known $$T^2(\mathcal{N}(T^{2+n}))=\mathcal{R}(T^2)\cap \mathcal{N}(T^n).$$
This implies that $\mathcal{R}(T^2)\cap \mathcal{N}(T^n)=\{0\}$, because $\mathcal{N}(T^2)=\mathcal{N}(T^{n+2})$.\\
Moreover, we know that $T^{-2}(\mathcal{R}(T^{2+n})=\mathcal{R}(T^n)+\mathcal{N}(T^2)$ and this equation gives us $L^{\Phi}(\mu)=\mathcal{R}(T^n)+\mathcal{N}(T^2)$, for some (equivalently, all) $n\geq 1$. This completes the proof.
\end{proof}
For a bounded linear operator $T$ on an arbitrary Banach space the Cesaro means is defined as
$$A_n(T)=\frac{I+T+T^2+...+T^{n-1}}{n}\ \ \ \ \ \ n\in \mathbb{N}.$$
Ergodic theory is concerned with the existence of the limit of the sequence $\{A_n(T)\}_{n\in \mathbb{N}}$ in various operator topologies. For investigation the convergence of this sequence one can use the following simple formulas(\cite{kr}, \cite{ly} and \cite{mb})\\
\begin{equation}\label{e2.1}
  \frac{T^n}{n}=\frac{n+1}{n}A_{n+1}(T)-A_n(T),
\end{equation}
\begin{equation}\label{e2.2}
  (I-T)A_n(T)=\frac{I-T^n}{n}
\end{equation}
\begin{equation}\label{e2.3}
  I-A_n(T)=(I-T)\frac{T^{n-2}+2T^{n-3}+...+(n-2)T+(n-1)I}{n}.
\end{equation}
 If we apply it for the bounded WCT operators $T=M_wEM_u$ on the Banach space $L^{\Phi}(\mu)$, then we have
$$A_n(T)=n^{-1}(I+M_{v_{n}}T), \ \ \ \ \ \ \forall n\in \mathbb{N}\setminus\{1\},$$

and $A_{1}(T)=I$, in which $v_n=\sum^{n-2}_{i=0}E(uw)^{i}$.\\
\begin{pro}
Let $|E(wu)|<1$ on $S(\Phi^{-1}(E(\Phi(w))))\cap S(\Psi^{-1}(E(\Psi(u))))$, $T=M_wEM_u$ be bounded WCT operator on the Orlicz spaces $L^{\Phi}(\Sigma)$. Then $\alpha(I-T)\leq 1$ and $\alpha(I-T^*)\leq 1$.
\end{pro}
\begin{proof}
Under our assumptions and by the Proposition \ref{p2.9} we get that $T$ is power bounded. So by the above observations we get that the sequence $\{\frac{T^n}{n}(f)\}$ tends to zero for every $f\in \mathcal{N}((I-T)^2)$. Hence by the Lemma 1.3  and Theorem 3.2 of \cite{gz} we get the result.
\end{proof}
Here, in the next theorem, we provide a dense subset of $L^{\Phi}(\mu)$ by means of the finite ascent and descent of bounded WCT operators.
\begin{thm}
Let $T=M_wEM_u$ be bounded WCT operator on the Orlicz spaces $L^{\Phi}(\Sigma)$ such that $\mathcal{R}(T^2)$ is closed. Then $\mathcal{R}(T^2)\cap \mathcal{N}(T^2)=\{0\}$ and $\mathcal{R}(T^2)+\mathcal{N}(T^2)$ is dense in $L^{\Phi}(\mu)$.
\end{thm}
\begin{proof}
For bounded WCT operator $T=M_wEM_u$(equivalently for $T^*$) we have obtained that its ascent is finite and specially $\alpha(T)=\alpha(T^*)\leq 2$. Hence by the Proposition \ref{p2.12} we get that $\mathcal{R}(T^2)\cap \mathcal{N}(T^2)=\{0\}$ and $\mathcal{R}(T^{*^2})\cap \mathcal{N}(T^{*^2})=\{0\}$. Since $\mathcal{R}(T^2)^{\perp}=\mathcal{N}(T^{*^2})$, then we get that
$$(\mathcal{R}(T^2)+\mathcal{N}(T^2))^{\perp}=\mathcal{R}(T^2)^{\perp}\cap \mathcal{N}(T^2)^{\perp}=\mathcal{R}(T^{*^2})\cap \mathcal{N}(T^{*^2})=\{0\}.$$
This implies that $\mathcal{R}(T^2)+\mathcal{N}(T^2)$ is dense in $L^{\Phi}(\mu)$.
\end{proof}
In the sequel we obtain that $L^{\Phi}(\mu)$  can be written as direct sum of its two closed sub-spaces.
\begin{pro}
Let $T=M_wEM_u$ be bounded WCT operator on the Orlicz spaces $L^{\Phi}(\Sigma)$ such that $|E(wu)|<1$ on $S(\Phi^{-1}(E(\Phi(w))))\cap S(\Psi^{-1}(E(\Psi(u))))$ and $\mathcal{R}(I-T)$ is closed. Then $L^{\Phi}(\mu)=\mathcal{R}(I-T)\oplus\mathcal{N}(I-T)$.
\end{pro}
\begin{proof}
Since $|E(wu)|<1$ on $S(\Phi^{-1}(E(\Phi(w))))\cap S(\Psi^{-1}(E(\Psi(u))))$, then by Proposition \ref{p2.9} we get that $T$ is power bounded. So $\{\frac{T^n}{n}\}$ is convergent in weak operator topology on $L^{\Phi}(\mu)$. So by Theorem 4.4 of \cite{gz} we have the result.
\end{proof}
Finally we have the following result.
\begin{thm}
Let $|E(wu)|<1$ on $S(\Phi^{-1}(E(\Phi(w))))\cap S(\Psi^{-1}(E(\Psi(u))))$ and $T=M_wEM_u$ . Then $\overline{\mathcal{R}(I-T)}=L^p(\mathcal{F})$ and equivalently we have the followings:\\
(i) $\mathcal{R}(I-T)=L^{\Phi}(\mu)$;\\
(ii) $I-T$ is invertible;\\
(iii) $\{\|(I-T)^{-1}(\frac{(n-1)I-M_{v_n}T}{n})\|\}_{n\in \mathbb{N}}=\{\|n^{-1}(M_{w_n}T+(n-1)I)\|\}_{n\in \mathbb{N}}$ is bounded;\\
(iv) $\{(I-T)^{-1}(\frac{(n-1)I-M_{v_n}T}{n})(f)\}_{n\in \mathbb{N}}=\{n^{-1}(M_{w_n}Tf+(n-1)f)\}_{n\in \mathbb{N}}$  converges for all $f\in L^{\Phi}(\mu)$. In which $w_n=\sum^{n-2}_{i=1}(n-i-1)E(uw)^{i-1}$.\\
In this case, $n^{-1}(M_{w_n}Tf+(n-1)f)\rightarrow (I-T)^{-1}(f)$ for all $f\in L^{\Phi}(\mu)$.\\
(v) $A_n(T)(f)$ converges to a $T$-invariant limit for all $f\in L^{\Phi}(\mu)$.
\end{thm}
\begin{proof} By our assumptions we get that $T$ is power bounded and so for every $f\in L^{\Phi}(\mu)$ we have $A_n(T)(f)\rightarrow 0$. As is defined in \cite{gz} $$B_n(T)=n^{-1}(T^{n-2}+2T^{n-3}+....+(n-2)T+(n-1)I).$$
 Hence for $T=M_wEM_u$ we have:
$$B_n(T)=n^{-1}(M_{w_n}T+(n-1)I),$$
 in which $w_n=\sum^{n-2}_{i=1}(n-i-1)E(uw)^{i-1}$.
 Therefore by Proposition 4.5 of \cite{gz} we have the proof of (i)-(iv). Since $\{\|E(uw)^n\|_{\infty}\}_{n\in \mathbb{N}}$ is uniformly bounded, then $T=M_wEM_u$ is power bounded. Also the closure of a norm bounded subset of $L^{\Phi}(\mu)$ is weakly compact. Hence it's weakly sequentially compact. Then we get (v).
\end{proof}

\end{document}